\documentclass[12pt]{article}
\usepackage{bbm}
\usepackage{graphicx}
\graphicspath{ {./images/} }
\usepackage{wrapfig}
\usepackage[mathscr]{euscript}
\usepackage{subcaption}
\usepackage{hyperref}
\usepackage{epstopdf}
\usepackage{enumerate}
\usepackage{amsmath,amsfonts,amssymb,amsthm,epsfig,epstopdf,titling,url,array}
\usepackage{color}
\usepackage{framed}
\usepackage[utf8]{inputenc}
\usepackage[english]{babel}
\usepackage{xparse}
\usepackage{authblk}
\usepackage[margin=2.4cm]{geometry}
\usepackage{authblk}
\newtheorem{thm}{Theorem}[section]

\newtheorem{lem}{Lemma}[section]
\newtheorem{prop}{Proposition}[section]

\newtheorem{rem}{Remark}[section]

\newtheorem{maintheorem}{Theorem}

\date{}
\begin{document}
	\title{On the dynamics of Halley's method}
	
	%-------adding authors--------------
	\author[1]{Gang Liu \footnote{liugangmath@sina.cn}}
	\author[2]{Soumen Pal \footnote{soumen.pal.new@gmail.com}}
	\author[2]{Saminathan Ponnusamy  \footnote{samy@iitm.ac.in}}
	\affil[1]{College of Mathematics and Statistics,
		Hengyang Normal University, China.}
	\affil[2]{Department of Mathematics,
		Indian Institute of Technology Madras, India}
	\date{}
	\maketitle
\begin{abstract}
In this article, we study the global dynamics of Halley's method applied to complex polynomials. Specifically, we analyze the structure and connectivity of the Julia set of this method. The convergence behavior, symmetry properties, and topological features of the corresponding Fatou and Julia sets are studied for various classes of polynomials, including unicritical, cubic, and quartic polynomials with non-trivial symmetry groups. In particular, we prove that Halley's method $H_p$ is convergent, its Julia set is connected, the immediate basins are unbounded, and the symmetry group of it coincides with that of the polynomial whenever $p$ belongs to one of the above classes. We further extend our results to a broader class of polynomials. It is shown that the immediate basin of Halley's method $H_p$ corresponding to a root of $p$ can be bounded. We also make some remarks on the dynamics of Halley's method applied to a cubic polynomial in general.		
\end{abstract}
	\textit{Keyword:} Fatou and Julia sets; Halley's method; Convergent classes; Symmetry in Julia set.\\
	AMS Subject Classification: 37F10, 65H05
	%=================================================
\section{Introduction}
Complex dynamics is the study of iteration of functions in the extended complex plane $ \widehat{\mathbb{C}}=\mathbb{C}\cup \{\infty\} $, mainly dealing with rational or entire functions. For a given rational function $ R: \widehat{\mathbb{C}} \to \widehat{\mathbb{C}} $, the dynamical behavior of the sequence \( \{R^n(z)\} \) as \( n \to \infty \) can exhibit intricate and rich structures, where $R^n$ denotes the $n$-fold composition of $R$ with itself. The \textit{Fatou set}  $ \mathcal{F}(R) $ consists of points in \( \widehat{\mathbb{C}} \) where the family \( \{R^n\} \) is normal in some neighborhood. Its complement, the \textit{Julia set} $ \mathcal{J}(R) = \widehat{\mathbb{C}} \setminus \mathcal{F}(R) $, represents the chaotic and highly sensitive dynamics of $ R $. These sets form the core objects of study in complex dynamics. By definition, the Fatou set is open, and a component of it is called a Fatou component. The Julia set is therefore closed. However, the Julia set is not always connected, and a connected subset of it known as a Julia component.
A fundamental aspect in characterizing these two sets is the classification of the fixed points of a rational function $R$. For a fixed point $z_0\in \widehat{\mathbb{C}}$ (i.e., $R(z_0)=z_0$), its multiplier $\lambda_{z_0}$ (or simply $\lambda$ whenever the suffix is not required) is defined by the complex number $R'(z_0)$ whenever $z_0\in \mathbb{C}$, or, whenever $R$ fixes $\infty$, its multiplier is $G'(0)$, where $G(z)=\frac{1}{R(1/z)}$. The fixed point $z_0$ is said to be \textit{attracting} if $|\lambda_{z_0}|<1$ (in a special case when $\lambda_{z_0}=0,$ $z_0$ is said to be \textit{superattracting}), repelling if $|\lambda_{z_0}|>1$, and \textit{indifferent} if $|\lambda_{z_0}|=1$. An indifferent fixed point is said to be \textit{parabolic} or \textit{rationally indifferent} if its multiplier is a root of unity, else, it is said to be \textit{irrationally indifferent}. A point $z^*$ is said to be a \textit{periodic point} with period $k$ if $R^k(z^*)=z^*$. However, the equality does not hold for $R^i$, for any $i=1,2,\dots, k-1$. A $k$-periodic point is classified as attracting, repelling, or indifferent by treating it as a fixed point of $R^k$.
Note that, an attracting fixed point $z_0$ of $R$ is contained in $\mathcal{F}(R)$, and there is a set $\mathcal{B}_{z_0}$ containing points whose successive iterations converge to $z_0$, i.e., $\mathcal{B}_{z_0}=\{z\in \widehat{\mathbb{C}}: \lim\limits_{n\to \infty}R^n(z)=z_0\}$. This set is called the \textit{basin of attraction} corresponding to $z_0$. The component of the basin that contains $z_0$ is known as the \textit{immediate basin}, and we denote it as $\mathcal{A}_{z_0}$. For a rational function $R$, a Fatou component $U$ can be of four types:
\begin{itemize}
	\item attracting domain if $U$ contains an attracting fixed point.
	\item parabolic domain if $\partial U$ contains a parabolic fixed point $\xi$, and $\lim\limits_{n \to \infty}R^n(z)=\xi$ for every $z\in U$.
	\item Siegel disk if there is an analytic homeomorphism $\phi: U \to \mathbb{D}$ such that $\phi \circ R \circ \phi^{-1}(z)=e^{2\pi i \theta}z$, for some irrational number $\theta$. Here $\mathbb{D}=\{z : |z| <	1\}$.
	\item Herman ring if there is an analytic homeomorphism $\phi: U \to \mathbb{A}_r=\{z : 1<|z| <r \}$ such that $\phi \circ R \circ \phi^{-1}(z)=e^{2\pi i \theta}z$, for some irrational number $\theta$.
\end{itemize}
A Fatou component is said to be a \textit{rotation domain} if it is either a Siegel disk or a Herman ring. Note that, for a complex function, there is another type of Fatou component, known as wandering domain. However, a rational function does not have a wandering domain. This is a remarkable result by Sullivan (see \cite[Theorem~8.1.2]{Beardon_book}). For a detailed study on complex dynamics, we refer to \cite{Beardon_book,CG1993,Milnor_book}.
%Root-finding algorithms like Newton's or Halley's methods, when interpreted as iterative processes, define rational maps whose dynamics can be analyzed through the lens of these sets. In this article, we investigate the rational function arising from Halley's method applied to complex polynomials and explore the corresponding Fatou and Julia sets in depth.
%

A root-finding method is defined as a map $ F $ that assigns to each polynomial $ p $ a rational function $ F_p $ such that every root of $ p $ is an attracting fixed point of $ F_p $.
 Iterative root-finding methods, such as Newton's and Halley's methods, not only provide practical schemes for locating roots of nonlinear equations but also give rise to interesting dynamical systems when studied in the complex plane. For a given point $z_0\in \widehat{\mathbb{C}}$ (often called an \textit{initial guess}), the behavior of its orbits, i.e., the sequence $\{F_p^i(z_0)\}_{i\in \mathbb{Z}}$, can be captured through the study of its Julia set $ \mathcal{J}(F_p)$ and Fatou set $ \mathcal{F}(F_p) $. This viewpoint connects numerical analysis with the field of complex dynamics (see \cite{ABP2004,BH2003,McMullen1987} and references therein). The \textit{order of convergence} of $F$ is $n$ (we say $F$ is \textit{$n$-th order convergent}) if the local degree of $F_p$ is at least $n$ at each simple root of $p$. Note that this notion describes the local rate of convergence near a root. For an example, Newton's method is quadratic convergent. Given a polynomial $p$, $F_p$ is said to be \textit{convergent} if the Fatou set $\mathcal{F}(F_p)$ is precisely the union of the basins of attraction of the roots of $p$, i.e., any point in the Fatou set will eventually converge to a root of $p$. This notion of convergence concerns the global behavior of the method in the complex plane. For a fixed natural number $d\geq 2$, a root-finding method $F$ is said to be \textit{generally convergent} if $F_p$ is convergent for all $p$ belonging to a dense subset of the family of all polynomials of degree $d$. Newton's method is generally convergent for $d=2$. McMullen proved that there is no root-finding method which is generally convergent for $d\geq 4$ (see Theorem 1.1, \cite{McMullen1987}).
%----------------

The dynamics of Newton’s method has been extensively studied and is now classical. Its rich literature includes \cite{HSS2001,Lei1997,Felex1989} and references therein. However, comparatively less is known about the global dynamics induced by higher-order methods such as Halley’s method. This may be attributed to the increased degree of the resulting rational map when a higher-order convergent root-finding method is applied, which makes the associated dynamical study more challenging. Given a complex polynomial $ p $, Halley's method is defined by the rational map
\begin{equation}\label{H_formula}
H_p(z) = z - \frac{2p(z)p'(z)}{2(p'(z))^2 - p(z)p''(z)}.
\end{equation}
This method is cubically convergent, i.e., the local degree of $H_p$ at every simple root of $p$ is at least three. Halley's method is the second member of the well-known family of root-finding methods, namely, K\"{o}nig's methods, which is derived by the following formula
$$K_{p,n}(z)=z+(n-1)\frac{\left(\frac{1}{p}\right)^{[n-2]}(z)}{\left(\frac{1}{p}\right)^{[n-1]}(z)},$$
where $n\geq 2$ is the order of convergence.
This also belongs to the Chebyshev-Halley family, which is derived by
$$
G^{\sigma}_p(z)=z-\left[1+\frac{1}{2}\dfrac{p(z)p''(z)}{(p'(z))^2-\sigma p(z)p''(z)}\right]\frac{p(z)}{p'(z)}.
$$
Note that each member of this family is a third-order convergent root-finding method, and $\sigma=\frac{1}{2}$ is Halley's method.
\par
In this paper, we extensively study the rational map $ H_p $  across various classes of polynomials $ p $ . Our goal is to analyze both the qualitative and quantitative characteristics of the Julia set $ \mathcal{J}(H_p) $ and the Fatou set $ \mathcal{F}(H_p) $. This research is motivated by our interest in exploring how iterative root-finding methods preserve the algebraic and the geometric properties of the original polynomials. For example, we investigate the existence of holomorphic Euclidean isometries that preserve the Julia sets of both $ p $ and $ H_p $. The collection of such maps, known as the \textit{symmetry group}, is discussed in detail in Section~\ref{Sect_sym} of this article. Note that this symmetry group is typically finite, exceptions happen whenever the Julia set is either the entire complex plane, a circle, or a line. However, for a polynomial $ p $ with at least two distinct roots, the Julia set of $ H_p $ cannot be $ \widehat{\mathbb{C}} $ or a circle, as the roots of $ p $ are attracting fixed points of  $ H_p $, while $ \infty $ is a repelling fixed point (see Proposition~\ref{prop_f.pts}). We establish that $ \mathcal{J}(H_p) $ is a line if and only if the polynomial $ p $ has precisely two distinct roots, each with equal multiplicity.
This result is presented in Theorem~\ref{exact2}.
\par
A significant portion of the paper is devoted to examining specific families of polynomials, including unicritical polynomials and certain cubic and quartic polynomials with special symmetries. For these classes, we show that the corresponding Halley’s method is convergent, the Julia set is connected, the immediate basins of attraction corresponding to the roots are unbounded, and the symmetry groups are identical (see Theorem~\ref{m1}). Note that, proving the Julia set connected is important in the context that unlike Newton's method, it is not straightforward to conclude about the connectivity of $\mathcal{J}(H_p)$ for every $p$. Moreover, in \cite{Honorato2013}, it is theoretically proved that the Julia set of K\"{o}nig's methods of order $n\geq 3$ is not always connected. Following the similar technique, Cumsille et al. proved the existence of a disconnected Julia set of Halley's method applied to the function $pe^q$, where $p$ and $q$ are polynomials (see \cite[Theorem 1.1]{CMHD2022}). Although, in both the cases, there is no explicit example of a function whose K\"{o}nig's method of order $n\geq 3$ has a disconnected Julia set.
\par
In Theorem~\ref{extn_gen}, we further extend our analysis to polynomials of the form $ p(z) = z(z^n + \lambda) $, establishing results on connectivity, the absence of rotation domains, and the preservation of symmetry. For this class of polynomials, we prove that an immediate basin of Halley's method can be bounded (see Theorem~\ref{extn_gen}(4)). Such property is totally different from Newton's method for polynomials. In fact, for Newton's method $N_p$ applied to a polynomial $p$, each of the immediate basins corresponding to the roots of $p$ is unbounded. Based on this and the connectedness of the Julia set when Newton's method applied to polynomials, Hubbard et al. \cite{HSS2001} found that there are  $ n $ accesses to $\infty$ in corresponding immediate basin of root, where $ n $ is the number of critical points in such immediate basin of root (counting with multiplicity). Thus, they can construct a finite set of points such that, for every root of every such polynomial, at least one of these points will converge to this root under Newton's map. These points are distributed in some circles with the same center (the origin). Recall that there exists bounded immediate basin of roots when Halley's method applied to polynomials. Therefore, it is very difficult to extend their method to the case of Halley's method, although the convergent speed of this method is faster. Recall that Newton's method and Halley's method are K\"{o}nig's methods of order $ 2 $ and $ 3 $, respectively. Identifying a polynomial for which an immediate basin of a root is bounded under K\"{o}nig's method of order $n>3$
seems to be a nontrivial task.
\par
We propose a parameterization of cubic polynomials which helps to prove the existence of $2$-periodic superattracting cycles (see Proposition~\ref{exst_2p}) as well as non-existence of a rotation domain under some hypothesis (see Proposition~\ref{no_rot_dom}).
\par
The results presented here generalize known facts about Newton’s method and provide a broader framework for analyzing rational maps derived from the iterative root-finding techniques. Our approach combines tools from complex analysis and topological dynamics.
\par
The structure of the paper is as follows: In Section~\ref{Sect-2}, we introduce Halley’s method and analyze its fixed points and properties. Section~\ref{Sect-3} studies the dynamics of $ H_p $ for various polynomial classes. Section~\ref{Sect-4} presents an extension to polynomials of higher degree, while Section~\ref{Sect-5} concludes with some results on Halley's method applied to cubic polynomials.
%%%-------------------------
\section{Halley's method: Properties and discussion}\label{Sect-2}
Recall that Equation~(\ref{H_formula}) gives Halley's method $H_p$ applied to a polynomial $p$. Whenever $p$ is a polynomial with exactly one root, say $p(z)=(z-a)^d$, the corresponding Halley's method is $H_p(z)=\frac{(d-1)z+2a}{d+1}$, which is a linear map. Thus, we consider polynomials having at least two distinct roots. Apart from the roots of $p$, there are fixed points of $H_p$, which are called \textit{extraneous}. The classification of fixed points of Halley's method is the following.
\begin{prop}\label{prop_f.pts}
A finite fixed point of $H_{p}$ is either a root of $p$ or a critical point of $p$ which is not a root of $p$.
More precisely,
\begin{enumerate}
\item a root $\alpha$ of $p$ with multiplicity $k$ is an attracting fixed point of $H_{p}$ with multiplier $\frac{k-1}{k+1}$;
\item if $\beta$ is a critical point of $p$ with multiplicity $l$ such that $p(\beta) \neq 0$, then $\beta$ is a repelling fixed point of $ H_p $ with multiplier $1+\frac{2}{l}$;
\item the point at $\infty$ is a repelling fixed point of $H_p$ with multiplier $\frac{d+1}{d-1}$.
\end{enumerate}
\end{prop}
The proof directly follows from Proposition 1 of \cite{BH2003}. We can observe that all extraneous fixed points of Halley's method are repelling.
A multiple critical point of a polynomial is called special if it is not a root of $p$. Such critical points, along with the distinct roots of $p$, determine the degree of $H_p$.
\begin{lem}
Suppose that a polynomial $p$ has $N$ many distinct roots and $s$ many special critical points with cumulative multiplicity $B$. Then $\deg(H_p)=2N+s-B-1$.
\end{lem}
\begin{proof}
Let $p$ have $m$ simple roots and $n$ multiple roots whose cumulative multiplicity is $A$. Therefore, $N=m+n$ and $d=m+A$, where $\deg(p)=d$. Let the number of simple critical points be $r$. As a multiple root of $p$ is also a critical point with multiplicity one less than the multiplicity of that point as a root of $p$, we have $r+B+(A-n)=d-1$. This gives that $r=N-B-1$. From Proposition~\ref{prop_f.pts} we know that all the fixed points of $H_p$ are simple, and apart from $\infty$, these are the roots of $p$ as well as the simple and special critical points of $p$. As the degree of a rational function is one less than the total number of fixed points (counting with multiplicity), we have $\deg(H_p)=N+r+s=2N+s-B-1$.
\end{proof}
\begin{rem}
If $p$ does not have any special critical point then $\deg(H_p)=2N-1$. If $p$ is unicritical then $\deg(H_p)=d+1$.
\end{rem}
\begin{prop}\label{ext_two}
Let $H$ be a rational map having exactly two repelling fixed points with equal multiplier. Then $H$ is conjugate to Halley's method applied to a polynomial $p$ if and only if $p$ is unicritical.
\end{prop}
\begin{proof}
First we consider that $H$ is conjugate to $H_p$, for some polynomial $p$ of degree $d$. Then, $p$ has exactly one critical point, say $z_0$, which is either a simple or a special critical point. By Proposition~\ref{prop_f.pts}, the multiplier of $z_0$ is $1+\frac{2}{l}$ for some natural number $l$. Therefore, by the hypothesis, we get $1+\frac{2}{l}=\frac{d+1}{d-1}$, which gives $l=d-1$. Thus, $p$ has a critical point with multiplicity $d-1$, and hence $p$ is unicritical.
\par
The converse part follows directly from Proposition~\ref{prop_f.pts}.
\end{proof}
\begin{rem}
The hypothesis in Proposition~\ref{ext_two} of having equal multiplier is necessary. For an example, if $p(z)=z^n(z-1)$, $n\geq 2$, then $\frac{n}{n+1}$ is the only finite extraneous fixed point of $H_p$, and its multiplier is $3$, which is not same as the multiplier of $\infty$.
\end{rem}
Note that for a polynomial $p$, if we construct another polynomial $q$ which is of the form $q(z)=c p(T(z))$, where $c$ is a non-zero constant and $T$ is a non-constant affine map, then $H_q(z)=T^{-1}(H_p(T(z)))$. This property is called Scaling property and the proof follows from
\cite[Theorem 2.2]{Nayak-Pal2022}.
\subsection{Maps preserving $\mathcal{J}(H_p)$}\label{Sect_sym}
A holomorphic Euclidean isometry is of the form $z\mapsto Az+B$, where $|A|=1$. For a rational function $R$, there may be such an isometry $\sigma$ which preserves the Julia set of $R$, i.e., $\sigma(\mathcal{J}(R))=\mathcal{J}(R)$. This type of maps plays a crucial role to investigate the dynamics of $R$. The collection of all holomorphic Euclidean isometries which preserve $\mathcal{J}(R)$ is called symmetry group, and is denoted by $\Sigma R$. In case of a polynomial, this group is well classified by Beardon (see \cite[Section 9.5]{Beardon_book}). A monic polynomial whose second leading coefficient is zero, is known as a normalized polynomial, and it can be represented as $z^\alpha p_0(z^\beta)$, where $p_0$ is a monic polynomial and $\alpha, \beta$ are non-negative integers, maximal for this expression. Then its symmetry group is $\{z\mapsto \lambda z: \lambda^\beta=1\}$ (see \cite[Theorem 9.5.4]{Beardon_book}). The generalization of maps which preserve the Julia set of a rational function can be found in \cite{Ferreira2023,NP2025}. The comparison between the symmetry groups of a polynomial and that of a root-finding method was initiated by Yang, who proved that if $p$ is a normalized polynomial which is not a monomial, then $\Sigma p\subseteq \Sigma N_p$, where $N_p$ is Newton's method applied to $p$ (see \cite[Theorem 1.1]{Yang2010}). His results were generalized by Liu and Gao \cite{Liu-Gao2015} by considering K\"{o}nig's methods. Further generalization was introduced by Nayak and Pal, by proving that if a root-finding method $F$ satisfies the Scaling property then for a non-monomial normalized polynomial $p$, $\Sigma p \subseteq \Sigma F_p$ (see \cite[Theorem 1.1]{Sym_dyn}). Since Halley's method belongs to K\"{o}nig's family, as well as it satisfies the Scaling property, for a normalized polynomial $p$ having at least two distinct roots, we have $\Sigma p\subseteq \Sigma H_p.$ We establish the equality in this relation for certain classes of polynomials. To do that we use the following result, which is Theorem C in \cite{NP2025}.
\begin{thm}[\cite{NP2025}] \label{CAOT}
Let $R$ be a rational function whose Julia set is neither $\widehat{\mathbb{C}}$ nor a circle nor a line segment. Further let $R$ be of the form
\begin{equation}
	R(z)=\gamma z^m\frac{P_1(z^{n_1})}{P_2(z^{n_2})},
	\label{form_coat}
\end{equation}
where $\gamma\neq 0$ is a constant; $P_1(z)$ and $P_2(z)$ are monic polynomials; and $m\geq 2$, $n_1$, $n_2$ are natural numbers, maximal for the expression. If $R$ does not have a parabolic or rotation domain, and $n=\gcd \{n_1,n_2\}\geq 2$ then $\Sigma R=\{z\mapsto \lambda z:\lambda^n=1\}$.
\end{thm}
\section{The dynamics of $H_p$}\label{Sect-3}
Since the structure of the Julia set of a rational function is usually complicated, determining the connectivity of this set is highly challenging. The first remarkable result on the connectivity of the Julia set of a rational function was given by Shishikura. He proved that whenever the Julia set is disconnected, the rational function has at least two weakly repelling fixed points (a fixed point which is either repelling or its multiplier is $1$) lying on two different components of the Julia set. The connectivity of the Julia set of Newton's method $N_p$ applied to a polynomial $p$ is is proved using this result, as $N_p$ has exactly one repelling fixed point at $\infty$, and all other fixed points are the roots of $p$, hence attracting. In the case of Halley's method, such a generalized statement cannot be made as $H_p$ always has at least two repelling fixed points. However, we can use the other aspect of Shishikura's theorem to study the connectivity of the Julia set of $H_p$. The precise statement is the following.
\begin{lem}\label{connected_J1}
Let $R$ be a rational function of degree at least two, and that there be a Julia component containing all the weakly repelling fixed points of $R$. Then $\mathcal{J}(R)$ is connected.
\end{lem}
Since for $H_p$, the point at $\infty$ is a repelling fixed point, we can implement \cite[Lemma 4.3]{Nayak-Pal2022} to derive the subsequent result.
\begin{lem}\label{bdry_pole}
The boundary of an unbounded immediate attracting basin of $H_p$ corresponding to a root of $p$ contains at least one pole.
\end{lem}
This result combined with \cite[Lemma 3.5]{Sym_dyn}, provides another tool to prove the connectivity of $\mathcal{J}(H_p)$.
\begin{lem}\label{all_pole}
If there is an unbounded Julia component of $H_p$ containing all the poles, then $\mathcal{J}(H_p)$ is connected.
\end{lem}
These two results are instrumental in proving the connectivity of $\mathcal{J}(H_p)$. Additionally, the following lemma is pivotal in investigating the real dynamics of $H_p$.
\begin{lem}\label{gen_dyn}
Suppose $R$ is a rational function with real coefficients having two real fixed points $x_{1}$ and $x_{2}\left(x_{1}<x_{2}\right)$ such that $R$ does not have any critical point, pole, or fixed point on $(x_{1}, x_{2})$. Then  the successive iterations $\{R^n(x)\}$ of any point $x$ in $(x_{1}, x_{2})$ will converge to either $x_{1}$ or $x_{2}$. More precisely, if $R(x)<x$ (or $R(x)>x$) then $\lim\limits_{n \to \infty} R^n(x)=x_1$ (or $\lim\limits_{n \to \infty} R^n(x)=x_2$, respectively).
\end{lem}
\begin{proof}
As $R$ has no critical point on $(x_1, x_2)$, $R$ is an increasing function on $(x_1, x_2)$. Moreover, as $R$ does not have any other fixed point on that interval, for all $x\in(x_1, x_2)$ either $R(x)<x$ or $R(x)>x$. First we consider $R(x)<x$ for all $x\in(x_1, x_2)$. Then for any $x\in (x_1,x_2)$ the sequence $\{R^n(x)\}$ is a decreasing sequence of real numbers which is bounded below by $x_1$. Thus, we have $\lim\limits_{n \to \infty} R^n(x)=x_1$. The other part can be proved similarly.
\end{proof}
For a rational function $R$, we call a Fatou component $U$ eventually lands on another Fatou component $V$ if there exists a natural number $k$ such that $R^k(U)=V$. In \cite[Lemma~3.3]{Sym_dyn} there is a criteria for a bounded Fatou component. The next result follows from that lemma.
\begin{lem}\label{bdd_FC}
Let $\mathcal{J}(H_p)$ be locally connected and $\mathcal{A}$ be an invariant Fatou component. Then, any Fatou component $U$ different from $\mathcal{A}$ that eventually lands on $\mathcal{A}$ is bounded.
\end{lem}
\begin{proof}
Since $\infty$ is a repelling fixed point of $H_p$, it follows trivially that if $\mathcal{A}$ is bounded then all its backward images are bounded. Now let $\mathcal{A}$ be unbounded. As $\mathcal{J}(H_p)$ is locally connected, $\mathcal{A}$ is simply connected, and $\infty$ is accessible from any interior point of $\mathcal{A}$ (i.e., for any point $z^*\in \mathcal{A}$, there is a simple arc $\gamma:[0,1]\to \mathcal{A}\cup \{\infty\}$ such that $\gamma(0)=z^*$ and $\gamma(1)=\infty$). Thus, it follows from the proof of \cite[Lemma 3.3]{Sym_dyn}, that if $U$ is a Fatou component different from $\mathcal{A}$ such that $H_p(U)=\mathcal{A}$, then $U$ is bounded. The rest is trivial as any Fatou component that lands on $U$ is bounded.
\end{proof}
Now we state the well-known Riemann-Hurwitz formula which is useful to determine the local degree of $H_p$ in an immediate basin.
\begin{lem}[Riemann-Hurwitz formula]
Let $R$ be a rational function that maps a Fatou component $\Omega_1$ to $\Omega_2$. Then $R:\Omega_1\to \Omega_2$ is a proper map of some
degree $ k $ and $ \mathcal{C}(\Omega_1) - 2 = k(\mathcal{C}(\Omega_2 ) - 2) + m $, where $ \mathcal{C}(.) $ denotes the connectivity of a domain and $ m $ is the number of critical points of $ R $ (counting with multiplicity) in $\Omega_1$.
\end{lem}
This formula asserts that if $\Omega_1$ and $\Omega_2$ are simply connected, then the local degree of $R$ in $\Omega_1$ is $m+1$.
Now we present the dynamical study of Halley's method for certain classes of polynomials. As $H_p$ is a third-order convergent method, a simple root of $p$ is a critical point of $H_p$ with multiplicity at least two. We call a critical point of $H_p$ is a \textit{free critical point} if it is not a root of $p$. These free critical points determine the dynamics of $H_p$. We first consider the polynomial $p$ having exactly two roots with equal multiplicity. It is known from \cite[Theorem 3]{Liu-Gao2015}, that $\mathcal{J}(K_{p,n})$ is a line. We provide here a simplified proof for $H_p$.
\begin{maintheorem}\label{exact2}
	The Julia set of $H_p$ is a line if and only if $p$ has exactly two distinct roots, each with the same multiplicity.
\end{maintheorem}
\begin{proof}
First we consider that $p$ is having exactly two roots with same multiplicity. Using the Scaling property we can consider $p(z)=\left(z^{2}-1\right)^{k}$. Then	$$
H_{p}(z)=\frac{z\left((2 k-1) z^{2}+3\right)}{(2 k+1) z^{2}+1}\qquad\text{and}\qquad H_{p}^{\prime}(z)=\frac{\left(4 k^{2}-1\right) z^{4}-6 z^{2}+3}{\left((2 k+1) z^{2}+1\right)^{2}}.
$$
The free critical points and the poles of $H_p$ are the solutions of
$ z^{2}=\frac{3 \pm 2i \sqrt{3} \sqrt{k^{2}-1}}{\left(4 k^{2}-1\right)}$ and $z^2=- \frac{1}{2k+1}$ respectively, which are non-real complex numbers.
As $\mathcal{J}(H_{p})$ is symmetric about both the axes, all critical points are in $\cup \mathcal{A}_{ \pm 1}$. As the origin is the only finite extraneous fixed point, it follows from Lemma~\ref{gen_dyn} that
$\mathcal{A}_{ \pm 1}$ are unbounded, and hence, their respective boundary contains at least one pole (by Lemma~\ref{bdry_pole}). Again, as poles are non-real, $\partial \mathcal{A}_{ \pm 1}$ contains both the poles. Hence, by Lemma~\ref{all_pole} we get that
$\mathcal{A}_{\pm 1}$ are simply connected. Using Riemann-Hurwitz formula, we get that the local degree of $H_p$ in each immediate basin is $3$, and therefore both the immediate basins are completely invariant. Hence, the Fatou set $\mathcal{F}(H_p)$ is the union of two immediate basins, and consequently, $\mathcal{J}(H_p)$ is a Jordan curve. Since the imaginary axis is invariant under the iteration of $H_{p}$ and both the poles are purely imaginary, as well as $\mathcal{J}(H_p)$ is preserved by the reflection about both the axes, we conclude that the Julia set is the imaginary axis.
\par
Conversely, let $\mathcal{J}(H_p)$ be a line.
Then $ p$ has exactly two roots, say $a$ and $b$, with multiplicity $k$ and $m$, respectively. Without loss of generality we can consider that $a$ and $ b$ are real.
Then $\mathcal{J}(H_p)$ is symmetric about the reflexion of real axis, and hence, $\mathcal{J}(H_{p})$ is a vertical line.
Again, using the Scaling property, we can consider the $\mathcal{J}(H_{p})$ as the imaginary axis. Note that the extraneous fixed point of $H_p$ is $\frac{b k+a m}{k+m}$, which is real and repelling. As $0 \in \mathcal{J}(H_p)$ is the only real number on the Julia set of $H_p$, we have
$
\frac{b k+a m}{k+m}=0$. This gives $b=-\frac{a m}{k} .
$
Therefore, we get that $ p(z)=(z-a)^{k}\left(z+\frac{a m}{k}\right)^{m} .
$ Consider the affine map $T(z)=a z$. Then $\mathcal{J}\left(H_{p}\right)=T\left(\mathcal{J}\left(H_{p}\right)\right)$, and hence, using the Scaling property we can consider
$
p(z)=(z-1)^{k}\left(z+\frac{m}{k}\right)^{m} .
$
Then $$H_{p}(z)=\frac{z\left(k(k+m-1) z^{2}+2(k-m) z+3 m\right)}{k(k+m+1) z^{2}+m}.$$
Note that, as $\mathcal{J}(H_{p})$ is the imaginary axis, therefore for every $y \in \mathbb{R}$, there exists $y^{*} \in \mathbb{R}$ such that
$H_{p}(i y)=i y^{*}.$
This can only be true if $k=m$.
Hence the result.
\end{proof}
It follows from the above theorem that $\Sigma H_p$ is an infinite set consisting of non-trivial translations whenever $p$ has exactly two roots with equal multiplicity. As $\infty$ is always a superattracting fixed point for a non-linear polynomial, $\Sigma p$ does not contain a translation. Therefore, the equality in $\Sigma p$ and $\Sigma H_p$ does not hold in this case. However, for a normalized polynomial with non-trivial symmetry group, having at least three distinct roots, the symmetry group of the corresponding Halley's method is non-trivial, finite, and consisting of rotations about the origin. Thus, the equality in these two symmetry groups can be expected.
\begin{maintheorem}\label{m1}
	Consider the following classes of normalized polynomials having at least two distinct roots:
	\begin{enumerate}
		\item unicritical
		\item cubic polynomials whose respective symmetry group is non-trivial
		\item quartic polynomials having a root at the origin and with non-trivial symmetry groups.
	\end{enumerate}
	Then the following are true:
	\begin{enumerate}[(i)]
		\item The Julia set of Halley's method is connected.
		\item The immediate basins corresponding to the roots of the polynomial are unbounded.
		\item The corresponding Halley's method is convergent.
		\item The symmetry groups of the polynomial and its corresponding Halley's method are identical.
	\end{enumerate}
\end{maintheorem}
A rational function is said to be \textit{geometrically finite} if the forward orbit of each critical point lying on the Julia set is finite. In \cite[Theorem A]{Lei1996}, Tan and Yin proved that for a geometrically finite map $R$, the Julia set is locally connected whenever $\mathcal{J}(R)$ is connected. To bring clarity in presentation the proof of Theorem~\ref{m1} is partitioned into different subsections.
\subsection{Unicritical polynomial}
%\begin{proof}[\textbf{Proof of Theorem~\ref{m1}(1)}]
\textbf{\textit{Proof of Theorem~\ref{m1}(1):}} Consider a normalized unicritical polynomial of the form $z^n+a$, where $a\in \mathbb{C}$. As we require a polynomial having at least two distinct roots, we may assume that $a\neq 0$. Then, using the Scaling property we consider $a=-1$. Therefore
$$H_p(z)=\frac{z((n-1)z^n+(n+1))}{(n+1)z^n+(n-1)} \qquad\text{and}\qquad H_p'(z)=\frac{(n^2-1)(z^n-1)^2}{((n+1)z^n+(n-1))^2}.$$
Thus, as the critical points of $H_p$ are the roots of $p$, there is no free critical point. Therefore each immediate basin is simply connected, and thus the Julia set is connected.
\par
Again, note that $H_p$ does not have any fixed point or critical point or pole in the interval $(1,\infty)$, and $H_p(x)<x$ for all $x\in (1,\infty)$. Thus, by Lemma~\ref{gen_dyn}, we get $(1,\infty)\subset \mathcal{A}_1$. This proves that $\mathcal{A}_1$ is unbounded. As the Julia set $\mathcal{J}(H_p)$ is preserved under rotations about the origin of order $n$, we conclude that every immediate basin is unbounded. Moreover, the Fatou set $\mathcal{F}(H_p)$ consists of the basins corresponding to the roots of $p$. Hence $H_p$ is convergent. Thus we conclude the assertions (i), (ii) and (iii) of Theorem~\ref{m1} for the case $1$.
\par
To prove (iv), we note that all the critical points of $H_p$ are in $\mathcal{F}(H_p)$. Therefore, $H_p$ is a geometrically finite map with connected Julia set. Hence $\mathcal{J}(H_p)$ is locally connected. Thus, from Lemma~\ref{bdd_FC}, we conclude that the immediate basins corresponding to the roots of $p$ are the only unbounded Fatou components of $H_p$. As these components are preserved by the rotations about the origin of order $n$, equality in the symmetry groups of $p$ and $H_p$ is satisfied.
\par
Figure~\ref{JS-1}(a) demonstrates the Fatou and Julia sets of $H_p$ applied to $p(z)=z^3-1$. The yellow region represents the basin of $1$, whereas the other two differently colored regions (green and blue) represent the basins of $\omega$ and $\omega^2$, respectively.
\begin{figure}[h!]
	\begin{subfigure}{.6\textwidth}
		\centering
		\includegraphics[width=0.78\linewidth]{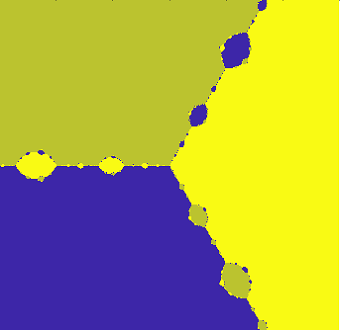}
		\caption{The Julia set of $H_p$, where $p(z)=z^3-1$}
%		\label{uni3}
	\end{subfigure}
	\hspace{-2.0cm}
	\begin{subfigure}{.6\textwidth}
		\centering
		\includegraphics[width=0.78\linewidth]{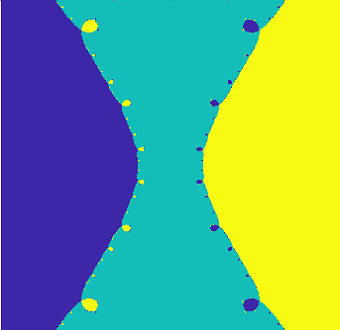}
		\caption{The Julia set of $H_p$, where $p(z)=z(z^2-1)$}
%		\label{J_1n2}
	\end{subfigure}
	\caption{The Julia set is the boundary of any two differently colored regions}
	\label{JS-1}
\end{figure}
%\end{proof}
\subsection{Cubic polynomial}
%\begin{proof}[\textbf{Proof of Theorem~\ref{m1}(2)}]
\textbf{\textit{Proof of Theorem~\ref{m1}(2):}} A cubic and normalized polynomial $p$ is of the form $z^3+a_1z+a_0$, where $a_0,a_1 \in \mathbb{C}$. If $p$ is not a monomial and $\Sigma p$ is non-trivial, then one of $a_0$ and $a_1$ is zero. First we consider $a_0=0$.
\begin{enumerate}[(I)]
\item Using the Scaling property, we consider $p(z)=z\left(z^{2}-1\right)$. Then Halley's method is
\begin{equation}\label{H_1n2}
H_{p}(z)=\frac{z^{3}\left(3 z^{2}+1\right)}{6 z^{4}-3 z^{2}+1} .
\end{equation}
Therefore, we have $$H_{p}^{\prime}(z)=\frac{3 z^{2}\left(z^{2}-1\right)^{2}\left(6 z^{2}+1\right)}{\left(6 z^{4}-3 z^{2}+1\right)^{2}}.$$
The free critical points of $H_p$ are $ \pm \frac{i}{\sqrt{6}}$, which are purely imaginary.
The poles of $H_p$ are $ \pm \sqrt{\frac{1}{12}(3 \pm i \sqrt{15})}$, which are neither real nor purely imaginary. The extraneous fixed points are $z_{-}=- \frac{1}{\sqrt{3}}$ and $z_{+}=\frac{1}{\sqrt{3}}$, and these are real.
Note that
\begin{equation}\label{phi_1n2}
H_{p}(i y)=\frac{i y^{3}\left(3 y^{2}-1\right)}{6 y^{4}+3 y^{2}+1}
=i\varphi(y) \text{ (say).}
\end{equation}
This shows that the imaginary axis is preserved under the iterations of $H_{p}$. As the free critical points are purely imaginary, they cannot converge to $\pm 1$.

Thus, there is exactly one critical point in each of $\mathcal{A}_{1}$ and $\mathcal{A}_{-1}$, and consequently, both $\mathcal{A}_{\pm 1}$ are simply connected.
\par
Note that $H_{p}'(x)>0$, for all
$ x \in \mathbb{R} \backslash\{0\} .
$ Moreover, $$H_{p}(x)-x=-\frac{x\left(x^{2}-1\right)\left(3 x^{2}-1\right)}{6 x^{4}-3 x^{2}+1}.$$
Therefore, $H_{p}(x)>x$ on $x \in(-\infty,-1) \cup\left(z_-, 0\right) \cup\left(z_+, 1\right)$
and $H_{p}(x)<x$ on $x \in\left(-1,z_-\right) \cup\left(0, z_+\right) \cup(1, \infty)$ (see Figure~\ref{Grh-1}(a)). It follows that
$ (-\infty,-1) \subset \mathcal{A}_{1} $, $\left(z_-,z_+\right) \subset \mathcal{A}_{0} $ and $ (1, \infty) \subset \mathcal{A}_{1}.$
%----MODIFY-----
This proves that $\mathcal{A}_{ \pm 1}$ are unbounded.
Therefore $\partial \mathcal{A}_{ \pm 1}$ contain at least one pole.
As $\mathcal{J}\left(H_{p}\right)$ is symmetric about both the axes, all poles are on $\partial \mathcal{A}_{1} \cup \partial \mathcal{A}_{-1}$.
As both the immediate basins are simply connected, all poles are lying on an unbounded Julia component. Hence $\mathcal{J}(H_p)$ is connected.
\par
To prove $\mathcal{A}_0$ is unbounded, we claim that the imaginary axis is in $\mathcal{A}_0$. This will prove consequently that $H_p$ is convergent. Recall the function $\varphi(y)=\frac{y^{3}\left(3 y^{2}-1\right)}{6 y^{4}+3 y^{2}+1}$ so that $\varphi^{\prime}(y)=\frac{3 y^{2}\left(y^{2}-1\right)\left(6 y^{2}-1\right)}{\left(6 y^{4}+3 y^{2}+1\right)^{2}}$.
We get that $0<\varphi^{\prime}(y)<1$ whenever $y \in(-\frac{1}{\sqrt{6}}, \frac{1}{\sqrt{6}})$.
Let $c_{+}=\frac{1}{\sqrt{6}}$ and $c_{-}=-\frac{1}{\sqrt{6}}$.
Then $\left(H_{p}\left(c_{+}\right), H_{p}\left(c_{-}\right)\right) \subset\left(c_{-}, c_{+}\right)$.
By the Contraction principle, we get that
$
\lim\limits_{n \rightarrow \infty} H_{p}^{n}\left(c_{+}\right)=0=\lim \limits_{n \rightarrow \infty} H_{p}^{n}\left(c_{-}\right) .
$
Note that $\varphi$ has two zeros, $z_{+}$ and $z_{-}$.
Consequently, we get that $ \lim \limits_{n \to \infty} H_{p}^{n}(y)=0$ for every $y \in\left(z_{-}, z_{+}\right)$ (see Figure~\ref{Grh-1}(b)).
Now for every point $\tilde{y}$ in $\left(-\infty, z_{-}\right)$ there is a natural number $ k$, smallest, such that $H_{p}^{k}(\tilde{y}) \in (z_{-1}, 0]$. Therefore, we get that
$\lim\limits _{n \to \infty} H_{p}(y)=0$ for all $ y \in(-\infty, 0]$, and hence, by symmetry, this holds for all $y$ in $(-\infty, \infty)$.
This shows that the imaginary axis, and thus all the free critical points, are contained in $\mathcal{A}_0$. Hence, there is no Fatou component other than the basins of $-1$, $0$ and $1$. Hence $H_p$ is convergent.
\par
We use Theorem~\ref{CAOT} to prove the equality between $\Sigma p$ and $\Sigma H_p$. Note that $H_p$ is of the form given in Equation~(\ref{form_coat}) and $n=2$. Moreover, as $H_p$ is convergent, there is no parabolic or rotation domain. Therefore, Theorem~\ref{CAOT} gives that $\Sigma H_p=\{Id, z\mapsto -z\}$, which is same as $\Sigma p$.
\begin{figure}[h!]
	\begin{subfigure}{.6\textwidth}
		\centering
		\includegraphics[width=0.80\linewidth]{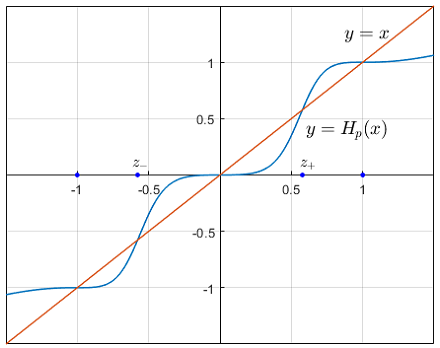}
		\caption{The graph of $H_p(x)$ for $x\in \mathbb{R}$}
%		\label{1n2}
	\end{subfigure}
	\hspace{-2.0cm}
	\begin{subfigure}{.6\textwidth}
		\centering
		\includegraphics[width=0.80\linewidth]{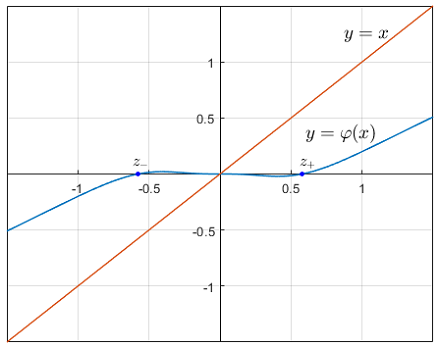}
		\caption{The graph of $\varphi(x)$ for $x\in \mathbb{R}$}
%		\label{phi2}
	\end{subfigure}
	\caption{The graphs of $H_p$ and $\varphi$ for $x\in \mathbb{R}$ as defined in Equations~(\ref{H_1n2}) and (\ref{phi_1n2}), respectively}
	\label{Grh-1}
\end{figure}
\par
Figure~\ref{JS-1}(b) portraits the dynamical plane of $H_p$ for $p(z)=z(z^2-1)$, where it can be seen that $\mathcal{J}(H_p)$ is preserved by the second order rotation about the origin.
\item If $a_1=0$ then $p$ is a unicritical polynomial, which is dealt in the previous sub-section.
\end{enumerate}
%\end{proof}
\subsection{Quartic polynomial}
%\begin{proof}[\textbf{Proof of Theorem~\ref{m1}(3)}]
\textbf{\textit{Proof of Theorem~\ref{m1}(3):}}	A normalized quartic polynomial is of the form $z^4+a_2z^2+a_1z+a_0$, where $a_i \in \mathbb{C}$ for $i=0,1,2$. If the origin is a root of this polynomial, then $a_0=0$. If there is a non-identity element in the symmetry group, then we have the following two cases: $z^4+a_1z$ and $z^4+a_2z^2$.
\begin{enumerate}
\item[(I)] First we consider the case when the polynomial is of the form $z^4+a_1z$. Then using the Scaling property we consider $p(z)=z\left(z^{3}-1\right)$. The corresponding Halley's method and its derivative are respectively given by
\begin{equation}
H_{p}(z)=z-\frac{z\left(z^{3}-1\right)\left(4 z^{3}-1\right)}{10 z^{6}-2 z^{3}+1}=\frac{3 z^{4}\left(2 z^{3}+1\right)}{10 z^{6}-2 z^{3}+1}\label{H_1n3}, 
\end{equation}
and
$$H_{p}^{\prime}(z)=\frac{12 z^{3}\left(z^{3}-1\right)^{2}\left(5 z^{3}+1\right)}{\left(10 z^{6}-2 z^{3}+1\right)^{2}}.$$
%$H_{p}(z)=z-\frac{z\left(z^{3}-1\right)\left(4 z^{3}-1\right)}{10 z^{6}-2 z^{3}+1}$
The solutions of $z^{3}=\frac{1}{4}$ are the extraneous fixed points of $H_p$, one of them is in the positive real axis, say $e_r$ and the other two are complex conjugates, say $e$ and $\bar{e}$. The
free critical points are the solutions of $z^{3}=-\frac{1}{5}$, among them, one is a negative real number, say $c_r$, while the other two, say $c$ and $\bar{c}$, are complex conjugates.
All poles of $H_p$ are non-real as these are the solutions of $z^{3}=\frac{1}{10}(1 \pm 3 i)$.
\par
We claim that $\cup_{i=1}^{3}\mathcal{A}_{i}$ does not contain a free critical point. As $H_p^m(c_r)\in \mathbb{R}$ for every $m\in \mathbb{N}$, if $c_r$ converges to a root of $p$, it should be either $0$ or $1$. However, as the Julia set $\mathcal{J}(H_p)$ is preserved by the rotations about the origin of order three, the immediate basin $\mathcal{A}_1$ cannot contain a negative real number. Hence $c_r\notin \mathcal{A}_{1}$. If $c\in \mathcal{A}_1$ then the reflection about the real axis asserts that $\bar{c}\in \mathcal{A}_1$. Again, note that $\sigma(\mathcal{A}_1)=\mathcal{A}_2$, where $\sigma(z)=e^{\frac{2\pi i}{3}}z$. Thus, if $c\in \mathcal{A}_1$ then $\sigma(c)=\bar{c}$ should be in $\mathcal{A}_2$, which is a contradiction. Thus, $\mathcal{A}_1$ does not contain a free critical point. Therefore, by symmetry, we conclude that free critical points do not lie on $\mathcal{A}_i$ , for every $i=1,2,3$. Thus, each of these domains contains exactly one critical point, and hence, these are simply connected.
\par
As $H_{p}^{\prime}(x)>0$ on the positive real axis, we get that
$H_{p}(x)$ is strictly increasing in $(1, \infty)$. Moreover, for every $x \in(1, \infty)$ we have $H_{p}(x)<x$ (see Figure \ref{Grh-2}(a)). It gives that $(1, \infty) \subset \mathcal{A}_{1}$. Therefore $\mathcal{A}_{1}$ is unbounded, and hence, $\partial \mathcal{A}_{1}$ contains at least one pole. As poles are non-real, and $\mathcal{J}\left(H_{p}\right)$ is symmetric about the $x$-axis as well as it is invariant under rotation about the origin of order three, all poles are lying on $\cup\partial\mathcal{A}_{i}$, where $i=1,2,3$. As every $\mathcal{A}_i$ is simply connected, all poles are lying on an unbounded Julia component. Hence,
$\mathcal{J}(H_{p})$ is connected.
\par
Consider the interval $(0,e_r)$. Then it $p$ has no critical point and $H_{p}(x)<x$ for all $x \in (0,e_r)$. Therefore, we get that $(0, e_{r}) \subset \mathcal{A}_{0} $. The critical value $c_r^*$ corresponding to the real free critical point $c_r$ is $c_{r}^*=-\frac{c_{r}}{5}$. Then we get $ 0<c_{r}^*<e_r$, and therefore, $H_{p}((-\infty, 0))=\left(-\infty, c_{r}^{*}\right)$.
Then for any $x \in(-\infty, 0)$, there exists $n_{x} \in \mathbb{N}$ such that $H_{p}^{n_x}(x)=0$ or $H_{p}^{n}(x) \in\left(0, c_{r}^{*}\right)$ for all $n \geq n_{x}$. Hence we have $(-\infty, e_r) \subset \mathcal{A}_{0}$. This proves that $\mathcal{A}_{0}$ is unbounded, and $c_r^*$ is in the immediate basin of $0$. Note that, $\mathcal{A}_0$ is preserved under the rotation about the origin of order $3$. Thus, all the free critical points are in $\mathcal{A}_0$. Hence, $H_p$ is convergent.
\par
The equality in $\Sigma p$ and $\Sigma H_p$ follows similarly as in the proof of Theorem~\ref{m1}(2) (see part (I)).
\par
The graph of $H_p$ on the real line, for $p(z)=z(z^3-1)$, is demonstrated in Figure~\ref{Grh-2}(a), where the red dot is the real free critical point $c_r$, the yellow dot is the critical value $c_r^*$, whereas the blue dots are the real non-zero fixed points of $H_p$. In Figure~\ref{JS-2}(a) equality in symmetry groups of $p$ and $H_p$ can be visualized.
\begin{figure}[h!]
	\begin{subfigure}{.6\textwidth}
		\centering
		\includegraphics[width=0.80\linewidth]{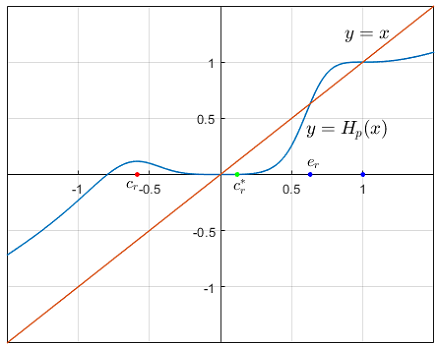}
		\caption{The graph of $H_p(x)$ for $x\in \mathbb{R}$}
%		\label{1n3}
	\end{subfigure}
	\hspace{-2.0cm}
	\begin{subfigure}{.6\textwidth}
		\centering
		\includegraphics[width=0.80\linewidth]{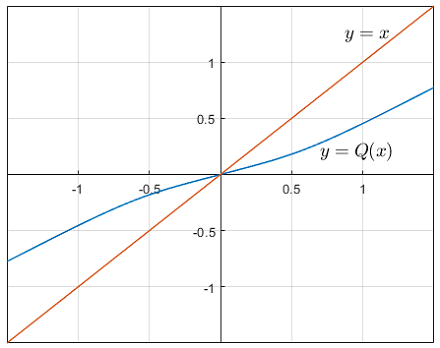}
		\caption{The graph of $Q(x)$ for $x\in \mathbb{R}$}
%		\label{plot_Q}
	\end{subfigure}
	\caption{The graphs $H_p$ and $Q$ for $x\in \mathbb{R}$ as defined in Equations~(\ref{H_1n3}) and (\ref{Q_form}), respectively}
	\label{Grh-2}
\end{figure}
\item[(II)] Now we consider the polynomial $z^4+a_2z^2$. Again using the Scaling property we consider $p(z)=z^2(z^2-1)$. Then $$H_{p}(z)=\frac{z\left(6 z^{4}-3 z^{2}+1\right)}{10 z^{4}-9 z^{2}+3}\qquad \text{and}\qquad H_{p}^{\prime}(z)=\frac{3\left(z^{2}- 1\right)^{2}\left(20 z^{4}-4 z^{2}+1\right)}{\left(10 z^{4}-9 z^{2}+3\right)^{2}}.$$
The poles of $H_p$ are $\pm \sqrt{\frac{1}{20}(9 \pm i \sqrt{39})}$ and the free critical points are $ \pm \sqrt{\frac{1}{10}(1 \pm 2 i)}$. As $\mathcal{J}(H_p)$ is symmetric under the reflections about both the axes, and free critical points are of the form $c, \bar{c}, -c $ and $-\bar{c}$, all the free critical points are in $\mathcal{A}_0$. Therefore, $\pm 1$ are the only critical points present in $\mathcal{A}_{ \pm 1}$ respectively, and hence, these are simply connected. It is easy to verify that $\mathcal{A}_{ \pm 1}$ are unbounded, and hence by symmetry, all the poles are on $\cup \partial \mathcal{A}_{ \pm 1}$. Therefore, we conclude that $\mathcal{J}(H_{p})$ is connected and the Fatou set is the union of basins of the roots of $p$. Hence $H_p$ is convergent.

Now we prove that $\mathcal{A}_0$ is unbounded. For $y\in \mathbb{R}$, we have
\begin{equation}\label{Q_form}
H_{p}(i y)=\frac{i y\left(6 y^{4}+3 y^{2}+1\right)}{10 y^{4}+9 y^{2}+3}=i Q(y).
\end{equation} Note that, the function $Q$ does not have any real critical point or real fixed point other than $0$. Moreover, $$Q(y)-y=-\frac{2y(y^2+1)(2y^2+1)}{10y^4+9y^2+3}.$$ Therefore, for any $y\in (0,\infty)$ we have $Q(y)<y$ (see Figure~\ref{Grh-2}(b)), and hence, from Lemma~\ref{gen_dyn} we get that $\lim\limits_{n\to \infty} Q^n(y)=0$. By symmetry we have $\lim\limits_{n\to \infty} Q^n(\tilde{y})=0$ for any $\tilde{y}\in (-\infty,0)$. This proves that the imaginary axis is in $\mathcal{A}_{0}$.
\par
Similar to the proof of Theorem~\ref{m1}(1), only unbounded components of $\mathcal{F}(H_p)$ are the immediate basins corresponding to the roots of $p$. Again, note that $\sigma(\mathcal{A}_0)=\mathcal{A}_0$ for every $\sigma \in \Sigma H_p$. Thus the immediate basin corresponding to a non-zero root of $p$ does not map to $\mathcal{A}_0$ under any $\sigma\in \Sigma H_p$. Moreover, the immediate basins $\mathcal{A}_{\pm 1}$ are symmetric to each other under the reflection about the imaginary axis. Hence, $\Sigma H_p=\{Id, z\mapsto -z\}=\Sigma p$, and this concludes the proof (see Figure~\ref{JS-2}(b)).
\begin{figure}[h!]
	\begin{subfigure}{.6\textwidth}
		\centering
		\includegraphics[width=0.78\linewidth]{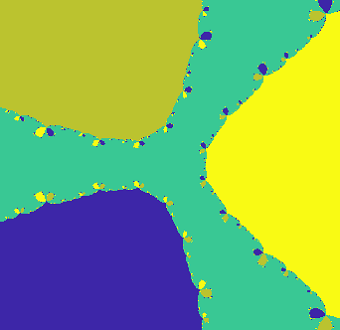}
		\caption{The Julia set  $\mathcal{J}(H_p)$ for $p(z)=z(z^3-1)$}
%		\label{H1n3}
	\end{subfigure}
	\hspace{-2.0cm}
	\begin{subfigure}{.6\textwidth}
		\centering
		\includegraphics[width=0.78\linewidth]{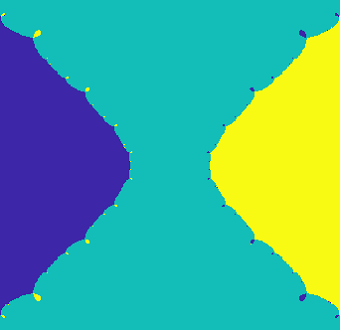}
		\caption{The Julia set  $\mathcal{J}(H_p)$ for $p(z)=z^2(z^2-1)$}
%		\label{2n2}
	\end{subfigure}
	\caption{The respective Julia set of $H_p$ is connected and immediate basins are unbounded}
	\label{JS-2}
\end{figure}
\end{enumerate}
%\end{proof}
%%%%%%%%%%%%%%%%%%%%%%%%%
\section{An extension to higher degree polynomials}\label{Sect-4}
The immediate possible extension of Theorem~\ref{m1} is to study the dynamics of $H_p$ for $p(z)=z(z^n-1)$, $n\geq 2$. In that theorem, the study has been done for $n=2,3$. Here is an attempt to extend Theorem~\ref{m1}, at least for partially.

\begin{maintheorem}\label{extn_gen}
For the polynomial $p(z)=z\left(z^{n}+a\right)$, where $a \in \mathbb{C} \backslash\{0\}$,
\begin{enumerate}
	\item the Julia set $ \mathcal{J}(H_p) $ is connected
	\item $H_p$ does not have any rotation domain
	\item $\Sigma p=\Sigma H_ p$
    \item each of the immediate basin of non-zero roots of $p$ is unbounded, but the immediate basin of $ 0 $ for $  H_p $ is bounded when $n\geq 7$.
\end{enumerate}
\end{maintheorem}
\begin{proof}
Using the Scaling property we consider $
p(z)=z(z^{n}-1)$. Then
\begin{equation}
H_{p}(z)=\frac{n z^{n+1}\left((n+1) z^{n}+(n-1)\right)}{(n+2)(n+1) z^{2 n}+(n+1)(n-4) z^{n}+2}\label{H_ext}
\end{equation}
and
$$
H_{p}^{\prime}(z)=\frac{n(n+1) z^{n}\left(z^{n}-1\right)^{2}\left((n+1)(n+2) z^{n}+2(n-1)\right)}{\left((n+2)(n+1) z^{2 n}+(n+1)(n-4) z^{n}+2\right)^{2}}.$$
We denote the immediate basins corresponding to the roots of $p$ by $\mathcal{A}_i$, $i=0,1,\dots, n$, where $\mathcal{A}_0$ and $\mathcal{A}_1$ are the immediate basins of $0$ and $1$, respectively, and $\mathcal{A}_{j}=\sigma^{j-1}(\mathcal{A}_1)$, for $j=2,3,\dots, n$ and $\sigma(z)=e^{\frac{2\pi i}{n}}z$.
\begin{enumerate}
\item The free critical points are the solutions of $$z^{n}=-\frac{2(n-1)}{(n+1)(n+2)}.$$
Hence whenever $n$ is even, there are two purely imaginary critical points and no real critical point, whereas, when $n$ is odd, there is exactly one real free critical point lying on the negative real axis. Using the similar argument used in the proof of Theorem~\ref{m1}(2) and (3) (see part (I) of each) we conclude that free
critical points are not in $\cup_{i=1}^{n} \mathcal{A}_i $. Therefore, there is exactly one critical point in $\mathcal{A}_{i}$, for every $i=1,2$, $\dots, n$. Hence, these Fatou components are simply connected.
\par
In order to prove that $ \mathcal{J} (H_p) $ is connected, we show that there is an unbounded Julia component containing all the poles or all the repelling fixed points.

The poles of $H_p$ are the solutions of
$$
z^{n}=\frac{-(n-4)(n+1) \pm n \sqrt{(n-7)(n+1)}}{2(n+1)(n+2)}.
$$
Note that for $n$ even, $H_p$ does not have a real pole. If $n=7$ then
$$
H_{p}(z)=\frac{7 z^{8}\left(4 z^{7}+3\right)}{\left(6 z^{7}+1\right)^{2}}.
$$
In this case, there is exactly one real pole, say $\xi_1$ (see Figure~\ref{Grh-3}(a)). Moreover, the multiplicity of each pole is $2$, and therefore, they coincide with free critical points. Hence all free critical points are in $ \mathcal{J} (H_p) $.
Whenever $n$ is even or $n<7$, all the poles are non-real.
% i.e., these are the solutions of
%$$
%z^{n}=\xi \text { or } z^{n}=\bar{\xi}
%$$
%for some $\xi$ with $\Im(\xi)\neq 0$.
In both the cases, the interval $(1, \infty)$ does not contain any fixed point, critical point or pole of $ H_p $. Moreover, $H_{p}(x)<x$ for all $x$ in $(1, \infty)$. Hence, by Lemma~\ref{gen_dyn}, we get $(1, \infty) \subset \mathcal{A}_{1}$. Thus, $\mathcal{A}_{1}$ is unbounded.
By the symmetry of the Julia set about the real axis as well as the rotation about the origin of order $n$ gives that all poles are in $\bigcup_{i=1}^{n} \partial \mathcal{A}_{i}$. As $\mathcal{A}_i$ is simply connected for all $i=1,2,\dots, n$, there is an unbounded Julia component containing all the poles. Therefore, $ \mathcal{J}(H_p) $ is connected.
\par
Now let $n$ be odd and $n>7$. In this case, $H_{p}$ has two real poles, both are in the negative real axis (see Figure~\ref{Grh-3}(b)). Therefore, using the same argument used in the previous paragraph, we conclude that $\mathcal{A}_{1}$ is unbounded.
% By symmetry we get that $\bar{\mathcal{A}}_{1}$ does not contain a negative real number. Therefore, negative real poles do not lie on $\partial \mathcal{A}_{1}$. Hence, if a pole $\xi$ is in $\partial \mathcal{A}_{1}$ then its corrugate $\bar{\xi}$, which is also a pole, is in $\partial \mathcal{A}_{1}$.
%$$
%H_{p}(z)=z-\frac{(n+1) z\left(z^{n}-1\right)\left(z^{n}-e_{p}^{n}\right)}{(\cdots \cdot-)}
%$$
%Let $e_{r}$ be the real extraneous fixed point lying on the positive real axis. 
Let $e_r=\sqrt[n]{1/(n+1)}$. Clearly, $p(e_r)\neq0$ and $p'(e_r)=0$.
By Proposition \ref{prop_f.pts}, we know that $e_r$ is a repelling fixed point lying in the Julia set of $H_p$.
Then there is no fixed point, critical point or pole in $ (e_r, 1) $. Moreover, $H_{p}(x)>x$ for all $x \in(e_r, 1)$. It follows from Lemma \ref{gen_dyn} that $e_{r} \in \partial \mathcal{A}_{1}$. As $\mathcal{A}_{i}$ is simply connected and unbounded, there exists an unbounded Julia component containing all the repelling fixed points. Hence $ \mathcal{J} (H_p) $ is connected.
\begin{figure}[h!]
	\begin{subfigure}{.6\textwidth}
		\centering
		\includegraphics[width=0.80\linewidth]{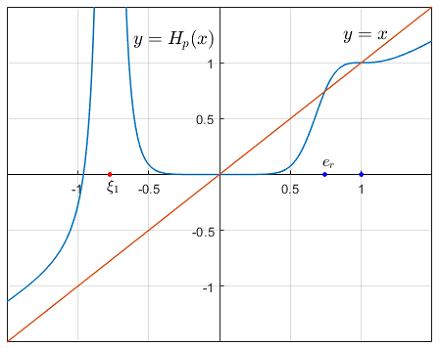}
		\caption{Exactly one real pole for $n=7$}
%		\label{plot_1n7}
	\end{subfigure}
	\hspace{-2.0cm}
	\begin{subfigure}{.6\textwidth}
		\centering
		\includegraphics[width=0.80\linewidth]{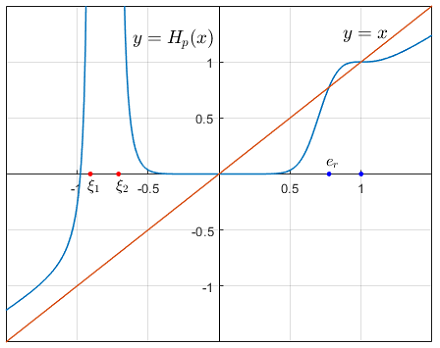}
		\caption{Two real poles for $n=9$}
%		\label{plot_1n9}
	\end{subfigure}
	\caption{The graph of $H_p(x)$ for $x\in \mathbb{R}$ as defined in Equation~(\ref{H_ext})}
	\label{Grh-3}
\end{figure}
\item Since the Julia set is connected, $H_p$ does not have a Herman ring. We need to verify the existence of a Siegel disk. Note that an invariant Siegel disk contains a rationally indifferent fixed point. As a fixed point of Halley's method applied to any polynomial is either attracting or repelling, Halley's method applied to a polynomial does not have a Siegel disk. Again, it is known from \cite[Theorem 11.17]{Milnor_book}, that the closure of the postcritical set contains the boundary of a Siegel disk. As $\lambda H_p(z)=H_p(\lambda z)$, where $\lambda^n=1$, the postcritical set is preserved by the rotations about the origin of order $n$. More precisely, the forward orbits of the free critical points of $H_p$ are contained in the straight lines passing through the origin, and the angle of two successive lines is $\frac{2\pi}{n}$. Since the boundary of a Siegel disk is a closed curve, it should contain $0$ and $\infty$. However, the origin is a superattracting fixed point of $H_p$ lying on the Fatou set. This contradicts the existence of a Siegel disk. Hence, $H_p$ does not have a rotation domain.
\item First we consider that $H_p$ does not have a parabolic domain. Thus from Theorem~\ref{CAOT} we get that $\Sigma H_{p}=\left\{z \mapsto \lambda z: \lambda^{n}=1\right\}=\Sigma p.$
\par
As free critical points are preserved by rotations about the origin of order $n$, if there is a parabolic domain then free critical points belong to it.
Hence all critical points are in the Fatou set, and therefore, $H_{p}$ is a geometrically finite map. As the Julia set is connected, it is locally connected.
It follows by Lemma~\ref{bdd_FC} that all the preimages of the immediate basins corresponding to the roots of $p$ are bounded.
Note that, $\left(e_r, \infty\right) \subset \mathcal{A}_{1}$ and $\mathcal{A}_{1}$ is symmetric about the real axis.
Since one of the free critical points is real and the real axis is preserved under the iteration of $H_{p}$, there exist parabolic periodic points on the real axis such that the immediate basin $U$ of these contains the real free critical point.

Therefore, for any $\sigma \in \Sigma H_{p},$ $\sigma (\mathcal{A}_{1})$ does
not map to $U$.
Hence $\sigma\left(\mathcal{A}_{1}\right)$ is mapped to $\mathcal{A}_{i}$, for some $i=1,2, \dots, n$.
This asserts that $\Sigma H_{p}=\left\{z \mapsto \lambda z: \lambda^{n}=1\right\}=\Sigma p.$
\par
\item %It follows from the proof of Theorem~\ref{extn_gen}(1)
Recall that $\mathcal{A}_1$ is unbounded and $\mathcal{A}_{j}=\sigma^{j-1}(\mathcal{A}_1)$ for $j=2,3,\dots n$ and $\sigma(z)=e^{\frac{2\pi i}{n}}z$. Thus,  each $\mathcal{A}_{j}$ is unbounded for $j\in\{1,2,\dots n\}$.
\par 
Next we will prove that $\mathcal{A}_0$ is bounded when $n\geq7$.
We distinguish two cases. Let us consider the case of odd number $n$.
Recall that $H_p$ has two sets of poles, say $S_1=\{z:z^n=\xi_1^n\}$ and $S_2=\{z:z^n=\xi_2^n\}$, where $\xi_1$ and $\xi_2$ are negative real numbers ($\xi_1=\xi_2$ whenever $n=7$). Note that all elements of $S_1$ and $S_2$ are preserved by the rotation about the origin of order $n$. As the immediate basin $\mathcal{A}_1$ is unbounded, its boundary contains at least one pole. We claim that $\mathcal{A}_1$ cannot meet the
negative real axis. It will follow subsequently that $\partial \mathcal{A}_1$ does not contain a real pole.
\par
Let there be a negative real number $r \in \mathcal{A}_1$. As $\mathcal{A}_1$ is simply connected, there exists a simple arc $\gamma$ joining $r$ and $1$. Note that $\gamma$ cannot lie entirely on the real axis, as both the intervals $(-\infty, r)$ and $(r, 1)$ contain points which are not in $\mathcal{A}_1$. Due to the symmetry of the Julia set $\mathcal{J}(H_p)$ about the real axis, there must exist another arc $\bar{\gamma}$ in $\mathcal{A}_1$ joining $r$ and $1$, which is the reflection of $\gamma$ about the real axis. (Note that $\gamma$ and $\bar{\gamma}$ may have common points if $\gamma$ contains some real values.) Then, some part of the closure of $\gamma \cup \bar{\gamma}$ forms a loop that surrounds the origin. This leads to a contradiction, because $\sigma(\gamma)$ and $\sigma(\bar{\gamma})$ are arcs in $\mathcal{A}_2$ ($ \sigma (z)=e^{\frac{2\pi i}{n}}z $), and the closure of their union would also surround the origin, which cannot happen due to the structure of the immediate basins.
\par
Let $\xi^*$ be a pole lying on the boundary of $\mathcal{A}_1$ such that $\Im(\xi^*)>0$. Then $\bar{\xi}^*\in \partial \mathcal{A}_1$. As $\xi_1$ and $\xi_2$ are real numbers, both of $\xi^*$ and $\bar{\xi}^*$ belong to the same set, either in $S_1$ or in $S_2$. Therefore, we get $\sigma(\bar{\xi}^*)=\xi^*$.
Note that the ray $L_1$ on the real axis joining $e_r$ (the real extraneous fixed point) and $\infty$ lies entirely in $\mathcal{A}_1$. Therefore, the rays $\sigma(L_1)$ and $\sigma^{n-1}(L_1)$ completely lie in $\mathcal{A}_2$ and $\mathcal{A}_n$, respectively. Thus, the boundary of $\mathcal{A}_1$ is enclosed between the rays $\sigma(\mathbb{R}^+)$, $\mathbb{R}^+$ and $\mathbb{R}^+$, $\sigma^{n-1}(\mathbb{R}^+)$, where $\mathbb{R}^+=(0,\infty)$. The boundary $\partial \mathcal{A}_1$ can be expressed as
$
\partial \mathcal{A}_1 = \Gamma_1 \cup \tilde{\Gamma}_1 ,
$
where $\Gamma_1$ contains the points $\xi^*$, $e_r$ and $\bar{\xi}^*$, whereas $\tilde{\Gamma}_1$ contains the points $\bar{\xi}^*$, $\infty$ and $\xi^*$ (i.e., $\Gamma_1$ is a bounded subset of the Julia set of $H_p$).

As $\sigma(\mathcal{A}_1)=\mathcal{A}_2$, its boundary contains $\xi^*$. Similarly, we can find another bounded subset $\Gamma_2$ of the Julia set that contains the points $\xi^*$, $\sigma(e_r)$ and $\sigma(\xi^*)$. Continuing the same way, we can get $\Gamma_n$ that contains $\sigma^{n-2}(\xi^*)$, $\sigma^{n-1}(e_r)$ and $\sigma^{n-1}(\xi^*)=\bar{\xi}^*$. Then $\cup_{i=1}^{n} \Gamma_i$ is a bounded subset of the Julia set that surrounds the origin. Hence $\mathcal{A}_0$ is bounded.
\par 
For even number $n$, the proof is similar.
The difference is the two sets of poles, say $S_1=\{z:z^n=-x_1\}$ and $S_2=\{z:z^n=-x_2\}$, where $x_1$ and $x_2$ are positive real numbers with $x_1\neq x_2$.
In addition, the assertion $e_{r} \in \partial \mathcal{A}_{1}$ can be easily deduced by the similar proof of connectivity of the Julia set when $n(>7)$ is odd.
\end{enumerate}
This concludes the proof.
\end{proof}
\begin{rem}
Although we prove that $H_p$ is convergent for $n=2,3$, we do not know whether the same is to be followed for higher $n$.
\end{rem}
\begin{rem}
It follows from the proof of Theorem \ref{m1} that the immediate basin of $ 0 $ for $  H_p $ is unbounded when $n=2,3$.
Such a result still holds for $n=4$, but it does not hold for $n\geq5$ (see \cite[Theorem 3.7]{CGJ2025}).
Thus, we provide another proof of boundedness of the immediate basin of $0$ when $n \geq 7$.
However, we do not know whether such a result can be proved by our method when $n=4,5,6$.
\end{rem}
\begin{rem}
For a polynomial $p$, unlike Newton's method $N_p$, in which each immediate basin corresponding to a root of $p$ is unbounded (see \cite[Proposition 6]{HSS2001}), Halley's method $H_p$ may have bounded immediate basins (see Figures \ref{JS-3}(a) and \ref{JS-3}(b).
Recall that Newton's method and Halley's method are K\"{o}nig's method of order $2$ and $3$, respectively.
For a given integer $n>3$, whether there exists a polynomial $p_n$ such that $K_{p_n,n}$ possesses a bounded immediate basin of root, where $K_{p_n,n}$ is K\"{o}nig's method
of order $n$ for polynomial $p_n$.
\end{rem}

\begin{figure}[h!]
	\begin{subfigure}{.6\textwidth}
		\centering
		\includegraphics[width=0.78\linewidth]{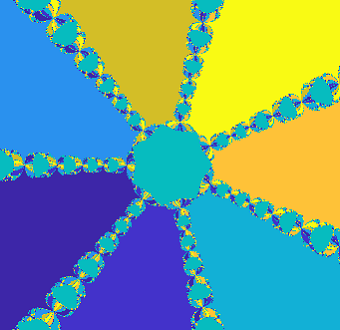}
		\caption{The Julia set  $\mathcal{J}(H_p)$ for $p(z)=z(z^7-1)$}
%		\label{H1n7}
	\end{subfigure}
	\hspace{-2.0cm}
	\begin{subfigure}{.6\textwidth}
		\centering
		\includegraphics[width=0.78\linewidth]{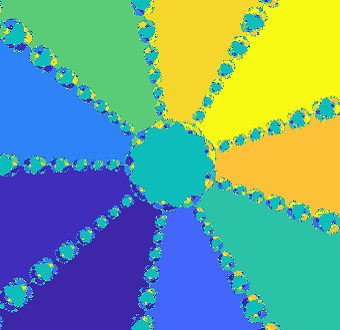}
		\caption{The Julia set  $\mathcal{J}(H_p)$ for $p(z)=z(z^9-1)$}
%		\label{1n9}
	\end{subfigure}
	\caption{The immediate basins corresponding to the non-zero roots of $p$ are unbounded, whereas $\mathcal{A}_0$ is bounded}
	\label{JS-3}
\end{figure}
\section{Illustrations on $H_p$ for cubic polynomials $p$}\label{Sect-5}
Theorem~\ref{m1} includes the dynamical study of Halley's method applied to cubic polynomials whose symmetry groups are non-trivial. We conclude the article by making some remarks on $H_p$ whenever $p$ is cubic and $\Sigma p$ is trivial.
\begin{lem}
If $p$ is cubic and has exactly two roots, then $H_{p}$ is convergent and $ \mathcal{J}(H_p) $ is a Jordan curve.
\end{lem}
\begin{proof}
Without loss of generality we may consider $0$ as a simple root of $p$ whereas $1$ is the multiple root with multiplicity two. Then $p(z)=z(z-1)^{2}$ and therefore,
$$
H_{p}(z)=\frac{3 z^{3}}{6 z^{2}-4 z+1} \qquad\text {and}\qquad H_{p}^{\prime}(z)=\frac{3 z^{2}\left(6 z^{2}-8 z+3\right)}{\left(6 z^{2}-4 z+1\right)^{2}}.
$$
The free critical points of $ H_p $ are $\frac{4 \pm i \sqrt{2}}{6}$. Since these are complex conjugate, the immediate basin $\mathcal{A}_{1}$ contains both the free critical points. It is similar to prove that $\mathcal{A}_0$ and $\mathcal{A}_{1}$ are unbounded, and as a consequence, the Julia set is connected. Using Riemann-Hurwitz formula we get that both the immediate basins are completely invariant. Hence the Julia set is a Jordan curve.
\end{proof}
As of now, we have classified certain polynomials whose Halley's method is convergent. If $H_p$ is not convergent for some polynomial $p$, then there will be a region where the successive iterations of each point does not converge to a root of $p$. We call such a region as \textit{bad region}. For Chebyshev's method applied to cubic or higher degree polynomials, such bad regions occur due to the presence of non-repelling extraneous fixed points (see \cite[Theorem 1.2]{Nayak-Pal2022}). However, as all the extraneous fixed points are repelling for Halley's method, the simplest way to find a bad region for $H_p$ is to construct a $2$-periodic superattracting cycle. In the following result we propose a parameterization of cubic polynomials to construct such cycles.
\begin{prop}\label{exst_2p}
	For $d=3$, Halley's method is not generally convergent.
\end{prop}
\begin{proof}
	As of now we have seen that $H_{p}$ is convergent whenever $p$ has exactly two roots of $\Sigma p$ is non-trivial. Now we consider $p$ is generic and $\Sigma p$ is trivial. Such types of polynomials can be parameterized as $p_b(z)=z^3+6z+b,$
where $b \in \mathbb{C} \backslash\{ \pm 4 i \sqrt{2}, 0\}$.
\par We denote Halley's method applied to $p_b$ by $H_b$.
Then
$$
H_{b}(z)=\frac{z^{5}-2 z^{3}-2 b z^{2}-2 b}{2 z^{4}+6 z^{2}-b z+12} \qquad \text {and} \qquad
H_{b}^{\prime}(z)=\frac{2\left(z^{3}+6 z+b\right)^{2}\left(z^{2}-1\right)}{\left(2 z^{4}+6 z^{2}-b z+12\right)^{2}} .
$$
Therefore, free critical points of $H_{b}$ are $\pm 1$. To prove $ H_b $ is not convergent, our aim is to construct a $ 2 $-periodic superattracting cycle of $ H_b $.

Note that $H_{b}(1)=\frac{1+4 b}{b-20}=\xi$ (say).
Therefore, if $H_{b}^{2}(1)=1$ then we get the equation
$$10 b^{6}-687 b^{5}+4326 b^{4}-24766 b^{3}-2569622b^{2}+885354 b-12815747=0,$$ and this gives
$(b+7) F(b)=0$,
where $$F(b)=10 b^{5}-757 b^{4}+9625 b^{3}-92141 b^{2}+388025 b-1830821.$$
Note that $b \neq-7$, unless the free critical point $ 1 $ becomes a root of $p$.
Thus, there are five values of $b$ for which $ H_b $ has a $2$-periodic superattracting cycle
$\{1, \xi\}$. Therefore, Halley's method is not generally convergent for cubic polynomials.
\end{proof}
Figure~\ref{JS-4}(a) shows the existence of superattracting $2$-periodic cycle for Halley's method applied $z^3+6z+62.5144396$. The enlarged image of one of the bad region is shown in Figure~\ref{JS-4}(b).
\begin{lem}
There are essentially five cubic polynomials for which the respective Halley's method has a superattracting $2$-periodic cycle.
\end{lem}
\begin{proof}
Note that $p_b(-z)=-p_{-b}(z)$ for any $b\in \mathbb{C}$. Thus, using the Scaling property, we get that $H_b$ and $H_{-b}$ are conjugate, more precisely, $H_b(-z)=-H_{-b}(z)$.
If we consider $H_{b}^{2}(-1)=-1$ then we get that $(b-7) F(-b)=0$. Here also $b \neq 7$, and as a consequence, we get another five values of $b$ for which $ H_b $ has a $ 2 $-periodic superattracting cycle $\{-1, \tilde{\xi} \}$, where $\tilde{\xi}=\frac{1-4 b}{b+20}$. However, if $H_b^2(1)=1$ for some $b$ then $H_{-b}^2(-1)=-1$.
Thus, there are essentially five distinct values of $b$ for which $H_{b}$ has a $2$-periodic superattracting cycle.
\end{proof}
\begin{rem}
Let $b$ be real. Then $\mathcal{J}(H_b)$ is symmetric under the reflection about the real axis. As $p_b^{\prime}(z)=3\left(z^{2}+2\right)$, extraneous fixed points of $H_b$ are $\pm i \sqrt{2}$.
Since $H_{b}(\bar{z})=\overline{H_{b}(z)}$, if there is an
unbounded simply connected Fatou component $V$ such that one of the extraneous fixed points is on the boundary of $V$ then there will be an unbounded Julia component containing all the repelling fixed points. Hence, by Lemma~\ref{connected_J1}, $\mathcal{J}(H_{p})$ is connected.
\end{rem}
As for every polynomial $p$, all the fixed point of $H_p$ are either attracting or repelling, it is trivial that $H_p$ cannot have an invariant Siegel disk. The next result discusses about the non-existence of a rotation domain for $H_b$ whenever $b$ is real.
\begin{prop}\label{no_rot_dom}
There is no rotation domain for $H_b$ whenever $b$ is real.
\end{prop}
\begin{proof}
For $b=0$ the Julia set of $H_b$ is connected (proved in Theorem~\ref{m1}).  Therefore, we consider $p(z)=z^3+6z+b$, where $b \in \mathbb{R} \backslash\{0\}$. Then both the free critical points of $H_p$ are real. As the real axis is preserved under the iterations of $H_b$, the forward orbits of these free critical points are in the real axis. The polynomial $p_b$ has a real root, thus the entire real line is not in the Julia set. Then the non-existence of a rotation domain follows from \cite[Theorem 11.17]{Milnor_book}, and the argument used in the proof of Theorem~\ref{extn_gen}(2).
\end{proof}
\begin{figure}[h!]
	\begin{subfigure}{.6\textwidth}
		\centering
		\includegraphics[width=0.78\linewidth]{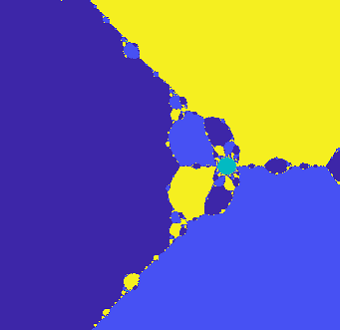}
		\caption{$\mathcal{J}(H_b)$ has bad regions for $b=62.5144396$}
%		\label{c2p}
	\end{subfigure}
	\hspace{-2.0cm}
	\begin{subfigure}{.6\textwidth}
		\centering
		\includegraphics[width=0.78\linewidth]{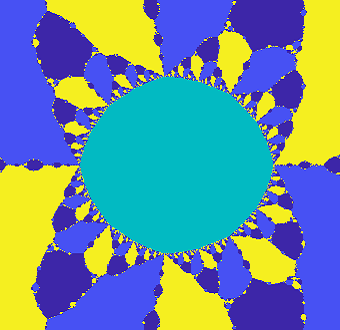}
		\caption{The enlarged bad region}
%		\label{z_cub}
	\end{subfigure}
	\caption{$H_b$ has a superattracting $2$-periodic cycle $\{1,5.905235\}$, its basins constitute bad regions}
	\label{JS-4}
\end{figure}
\textbf{Acknowledgement:} Soumen Pal is supported by Indian Institute of Technology Madras through a Postdoctoral Fellowship.
%--------------

\end{document}